\documentclass[11pt,twoside]{article}
\usepackage{pgfplots}
\pgfplotsset{compat = 1.7}
\usepackage{epsfig,epstopdf}
\usepackage{bbm}
\usepackage{graphicx}
\usepackage{amsthm,amsmath,amssymb,amsfonts}
\usepackage[parfill]{parskip} % Activate to begin paragraphs with an empty line rather than an indent
\usepackage{hyperref,url,mathrsfs} 
\newtheorem{theorem}{Theorem}
\newtheorem{proposition}{Proposition}
\newtheorem{lemma}{Lemma}

\title{{On the stability of an adaptive learning dynamics \\ in traffic games}
\thanks{This work was partially supported by FONDECYT 1130564 and Complex Engineering Systems Institute, ISCI (ICM-FIC: P05-004-F, CONICYT: FB0816).}}

\author{Miguel A. Dumett
\thanks{Computational Science Research Center, San Diego State University, California, USA; (mdumett@sdsu.edu).} 
\and Roberto Cominetti 
\thanks{Facultad de Ingenier\'ia y Ciencias, Universidad Adolfo Ib\'a\~nez, Santiago, Chile; (roberto.cominetti@uai.cl).}
}
\date{\today}

\long\def\drop#1{}

\begin{document}

\maketitle

\begin{abstract}
This paper investigates the dynamic stability of an adaptive learning procedure in a 
traffic game. Using the Routh-Hurwitz criterion we study the stability of the rest points 
of  the corresponding mean field dynamics. In the special case with two routes and two 
players we provide a full description of the number and nature of these rest points as well 
as the global asymptotic behavior of the dynamics. Depending on the parameters of the model, 
we find that there are either one, two or three equilibria and we show that in all cases the mean 
field trajectories converge towards a rest point for almost all initial conditions.
\end{abstract}

%\begin{keywords}

%\begin{AMS}
%\end{AMS}

\pagebreak

\section{Introduction}
Traffic in congested networks is frequently modeled as a game among drivers, with equilibrium   interpreted as a 
steady state  that emerges from some unspecified adaptive mechanism in driver behavior. 
This is the case in Rosenthal's model where drivers are taken 
as individual players  \cite{ros}, in Wardrop's non-atomic equilibrium with traffic modeled by continuous flows  
\cite{rwar}, and  in {stochastic equilibrium} 
with routing decisions based on random choice models  \cite{dag,fis}. 
Empirical evidence for the existence of some adaptive 
mechanism that leads to  equilibrium has  been presented in \cite{avi,hor,mkp,sel}.

There is a vast literature dealing with adaptive dynamics 
in repeated games.
The most prominent procedure is {\em fictitious play} which assumes that at each 
stage players choose  a best response to the empirical distribution of past moves 
by their opponents  \cite{bro,rob}, or a smooth best response using a Logit 
random choice \cite{fl,hsa}. For large games this can be very demanding as it 
requires to monitor the moves of all players. A milder assumption 
is that  players observe only the payoff obtained at every stage. Procedures such 
as exponential weight \cite{fr,acf}, calibration \cite{fv,fv2}, and no-regret \cite{han,har,hm}, 
deal with such limited information situations: players build statistics of their past 
performance and infer what the outcome would have been if a different strategy 
had been played. Eventually, adaptation leads to a steady state  in which 
no player regrets the choices she makes. For a complete account of these dynamics 
we refer to the monographs \cite{fl,Yo}.

A simpler discrete time adaptive process was considered in \cite{cominetti}  using only the  sequence of 
realized payoffs. The idea is similar to  {\em reinforcement dynamics} 
\cite{ar,beg,bos,er,las,po}, though it differs in the way the state is defined 
as well as in how this state affects the player's decisions. 
%Convergence of reinforcement dynamics has been established for the case of an
%i.i.d. environment for zero-sum games and also for some games with
%unique equilibria \cite{las,beg}. 
Namely,  at every stage $n\in\mathbb{N}$ each player $i\in I$ selects a route $r^i_n\in R$ 
at random with Logit probabilities  $\pi^{ir}_n\!\!=\!\pi^{ir}\!(x^i_n)$
(see \eqref{adp} in the next section) where the vector $x^i_n=(x^{ir}_n)_{r\in R}$ is a state variable 
that represents player $i$'s estimates of the average travel times on all possible routes $r\in R$. 
The collective choices of all the players determine the random load $u^r_n$ of each route 
and the corresponding travel times $C^r_{u^r_n}$. Each player $i$ then observes the travel time of the 
route  $r=r^i_n$ that was chosen and updates her estimate  for that particular route as a weighted average 
between $x^{ir}_n$ and the observed travel time $C^r_{u^r_n}$.  
This procedure is repeated day after day, generating 
a discrete time stochastic process.

A basic question is whether this simple adaptive mechanism induces coordination among players 
and leads towards an equilibrium. A partial answer was provided in \cite{cominetti} 
using results in stochastic approximation \cite{ben,kyi} by studying the associated mean field dynamics 
\begin{equation}
\dot x^{ir}= \pi^{ir}\!(x^i)(C^{ir}\!(x^{\mbox{-}i}) - x^{ir})\quad \quad \forall\, i \in I, r \in R, \label{adintro}
\end{equation}
where $C^{ir}\!(\cdot)$ is the expected travel time of route $r$ conditional on the event that
player $i$ chooses this route (a detailed description is given in the next section).
The rest points of \eqref{adintro} turn out to be Nash equilibria for 
an underlying perturbed game.  Moreover, when the Logit parameters are small  enough there is a unique
rest point which is a global attractor, and the stochastic process 
converges almost surely towards this equilibrium \cite{cominetti}. 

In this paper we explore what happens for larger values of the Logit parameters
in terms of the number and nature of the rest points of the mean field dynamics \eqref{adintro}.
The paper is organized as follows. 
In \S\ref{sec2} we review the  model  for  the adaptive dynamics in the traffic game, 
including a precise description of the stochastic process and the mean field dynamics.
In \S\ref{sec3} we discuss the Routh-Hurwitz criterion for stability of equilibria, 
providing an explicit expression for the Jacobian of \eqref{adintro}.
Our main results are presented in \S\ref{sec4} and deal with the 
asymptotic stability of equilibria for a $2\times 2$ traffic game.
Specifically, in \S\ref{sec41}  we exploit the Routh-Hurwitz criterion
to derive a simple necessary and sufficient condition for the stability of rest points. 
Then, in \S\ref{sec42} we show that there can be only one, two or three equilibria, and we 
determine their stability. Next, in  \S\ref{sec43} we show that 
the mean field dynamics are $K$-monotone for a suitable order $\leq_K$,
from which we deduce that  the forward orbits of \eqref{adintro} converge towards
a rest point for almost all initial conditions. 
%In \S\ref{secstoc} we use these results to establish the almost sure convergence of the stochastic adaptive dynamics.
Finally, in \S\ref{sec44} we 
consider the case in which both players are identical. In this case there is a 
special symmetric equilibrium whose  stability can be characterized
explicitly  in terms of the parameters defining the model, and we show 
that when the symmetric rest point becomes unstable there appear 
two additional non-symmetric stable equilibria.

\section{Adaptive dynamics in a simple traffic game}\label{sec2}
Consider a   set of routes $R=\{1,\ldots,M\}$ which are used concurrently by a finite set of drivers $I=\{1,\ldots,N\}$.
The travel time of each route $r\in R$ depends on the number of drivers on that route and is given by a
non-decreasing sequence 
\begin{equation}
C^r_1 \le C^r_2 \le \cdots \le C^r_N. \label{cis}
\end{equation}
Suppose that each driver $i\in I$ chooses a route $r^i\in R$ at random with Logit probabilities
\begin{equation}
\mathbb{P}(r^i\!=\!r)=\frac{ \exp(- {\beta}_i x^{ir})}{ \sum_{s \in R} \exp(- {\beta}_i x^{is})}\triangleq \pi^{ir}\!(x^i)\quad\quad\forall\,r\in R\label{adp}
\end{equation}
where the vector $x^i=(x^{ir})_{r\in R}$ describes the driver's {\em a priori} estimate of the travel time of each 
possible route. 
The parameter $\beta_i>0$ is a characteristic of the player and 
captures how sensitive is the driver to differences of costs: for $\beta_i$ small the 
probability distribution $\pi_i(x^i)$ is roughly uniform over $R$, whereas for $\beta_i$ large it assigns
higher probabilities to routes with smaller  travel time estimates $x^{ir}$.

Let $Y^{ir}\!\!=\!\mathbbm{1}_{\{r^i=r\}}$ be the random variable indicating whether driver $i$ selects route $r$. 
The {\em load} of route $r$ is $u^r\!=\!\sum_{i\in I}Y^{ir}$ (the number of drivers that have chosen to utilize that route) which induces a random 
travel time $C^r_{u^r}$. Hence, while the route choice of driver $i$ is based solely on her own
estimate $x^i$, the route travel times depend on the collective choices of all players.

Following \cite{cominetti}, we consider a dynamical model in which the estimates $x^{ir}$ evolve 
in discrete time as a weighted average of travel times experienced in the past.
Formally, on day $n\in \mathbb{N}$ each driver $i\in I$ chooses
a route $r^i_n\in R$ at random according to the Logit probabilities \eqref{adp} based on her current
estimate vector $x_n^i$.  After observing the travel time of the chosen route $r^i_n$,
player $i\in I$ updates her estimate for this route  as a weighted average between the previous estimate and the new observation, 
keeping unchanged the estimates for the routes that were not observed, namely
\begin{equation}\label{discrete_dynamics}
x^{ir}_{n+1}=\left\{\begin{array}{cl}
(1\!-\!\alpha_n)x^{ir}_n+\alpha_n C^r_{u^r_n}&\mbox{for }r=r^i_n\\
x^{ir}_n&\mbox{for }r\neq r^i_n\end{array}\right.
\end{equation}
where $\alpha_n\in(0,1)$ is the weight assigned to the new observation. 
This discrete time Markov process  can be written as a Robbins-Monro process
\begin{equation}\label{discrete_dynamics2}
\mbox{$\frac{x^{ir}_{n+1}-x^{ir}_n}{\alpha_n}$}=Y^{ir}_n(C^r_{u^r_n}-x^{ir}_n)\quad\quad\forall\,i\in I, r\in R
\end{equation}
which can be analyzed using the ODE method in stochastic approximation. Indeed, under the mild conditions 
$\sum_{n\geq 0}\alpha_n=\infty$ and $\sum_{n\geq 0}\alpha_n^2<\infty$, the asymptotic behavior
of this stochastic process is closely related to the asymptotics of the mean field  dynamics
\begin{equation}\label{mean_field}
\dot x^{ir}=\mathbb{E}_\pi[Y^{ir}(C^r_{u^r}-x^{ir})]\quad\quad\forall\,i\in I, r\in R
\end{equation}
where $\mathbb{E}_\pi[\cdot]$ denotes expectation with respect to the probability distribution
induced by the player's choice probabilities $\pi_i=\pi_i(x^i)$ for all $i\in I$.
Specifically, if \eqref{mean_field} has a global attractor $x^*\in\mathbb{R}^{MN}$ then
the stochastic process \eqref{discrete_dynamics} converges almost surely towards $x^*$, whereas 
a local attractor has a positive probability of being attained as the limit of $x_n$. This raises  
the question of characterizing the local attractors of the mean field dynamics \eqref{mean_field}.

Before proceeding let us introduce some notation.
We  denote $x=(x^i)_{i\in I}$  the vector 
of player's estimates, and $x^{\mbox{-}i}=(x^j)_{j\neq i}$ the estimates of all the drivers except $i$.
The vector field defining \eqref{mean_field} can be expressed more
explicitly by conditioning on the random variable $Y^{ir}$. 
Denoting $u^r_{\mbox{-}i}=\sum_{j\neq i}Y^{jr}$ we have
\begin{equation}
\dot x^{ir}= \pi^{ir}\!(x^i) (C^{ir}\!(x^{\mbox{-}i}) - x^{ir})\quad \quad \forall\, i \in I, r \in R, \label{ad}
\end{equation}
where
\begin{equation}\label{expected}
C^{ir}\!(x^{\mbox{-}i})=\mathbb{E}_\pi[C^r_{u^r}|Y^{ir}\!=\!1] =\mathbb{E}_\pi[C^r_{1+u_{\mbox{-}i}^r}].
\end{equation}
Note that the latter  depends only on the choice probabilities $\pi_j^r=\pi_j^r(x^j)$ of the drivers  $j\neq i$,
namely
\begin{eqnarray}\nonumber
C^{ir}\!(x^{\mbox{-}i})&=& \sum_{u = 0}^{N-1} C^r_{1+u}\;\mathbb{P}(u^r_{\mbox{-}i}\!=\!u)\\
&=&\sum_{u = 0}^{N-1} C^r_{1+u} \!\!\!\!\!\sum_{\tiny{\begin{array}{c} A\!\subseteq\! I \!\setminus\! \{ i \}\\|A| = u \end{array}}} \!\!\prod_{j \in A} {\pi}_j^r \prod_{j \notin A} (1\! -\! {\pi}_j^r).\label{adc}
\end{eqnarray}
%where for simplicity we omitted the variable $\pi_j^r=\pi_j^r(x^j)$ in the choice  probabilities.

\section{Rest points and dynamic stability}\label{sec3}
Let $S$ denote the set of rest points of the  dynamics \eqref{ad}, that is to say, 
the solutions of the  fixed-point equations
\begin{equation}
x^{ir}=C^{ir}\!(x^{\mbox{-}i})\quad\quad \forall\, i \in I, r \in R. \label{eq}
\end{equation}

The existence of rest points follows directly from Brower's fixed point theorem, for all possible values of the parameters $\beta_i$ and $C^r_u$.
From a structural viewpoint we note that \eqref{eq} involves only polynomial and exponential functions 
and therefore $S$ is a definable set in the $o$-minimal structure ${\mathbb R}_{\mathop{\rm alg},\exp}$  known 
as the {\em real exponential field} \cite[Example 1.7]{dries}.  It follows that $S$ has finitely many connected components.
In fact, from  \cite[Theorem 3.12]{coste}, there is an integer $k$ which depends only on $N$ and $M$ such that  $S$ has 
at most $k$ connected components for all possible values of the parameters. Moreover, for each $i\in I$ and $r\in R$ the projection
$S^{ir}=\{x^{ir}: x \in S\}$ onto the real line is a finite union of intervals and points. 

Now,  according to  \cite[Theorem 10]{cominetti}, if we have $\omega\Delta<2$ where 
 \begin{eqnarray*}
 \omega&=&\max_{i\in I}\mbox{$\sum_{j\neq i}\beta_j$}\\
 \Delta&=&\max_{r\in R}\max_{u=2,\ldots,N} C^r_{u}-C^r_{u-1}
 \end{eqnarray*}
then \eqref{ad} has a unique rest point $x^*$ which is moreover a global attractor for the 
continuous time dynamics, and the stochastic process \eqref{discrete_dynamics} satisfies $x_n\to x^*$ almost surely.
The condition $\omega\Delta<2$ holds when either the $\beta_i$'s are small (i.e. the players are 
mildly sensitive to travel time differences) or the congestion jump $\Delta$ is small (i.e.  travel times
do not vary too much with increasing congestion). In this paper we are concerned with the study of the dynamic stability 
of the rest points of  the mean field  dynamics  \eqref{ad} beyond this regime when $\omega\Delta\geq 2$.

\subsection{The Jacobian of \eqref{ad} and Routh-Hurwitz stability}
Let $G:\mathbb{R}^{NM}\to\mathbb{R}^{NM}$ with $G^{ir}(x)=\pi^{ir}\!(x^i) (C^{ir}\!(x^{\mbox{-}i}) - x^{ir})$ 
be the vector field that defines the mean field dynamics \eqref{ad}. A rest point 
$x\in S$ is a linearly stable local attractor if all the eigenvalues of the Jacobian matrix  
$$J(x) = \left( \frac{\partial G^{ir}}{\partial x^{ks}}(x) \right)_{i,k \in I; r,s \in R}$$
have strictly negative real part.
The next Lemma provides an explicit expression for this Jacobian.
For notational convenience we will omit the variables
and write $\frac{\partial G^{ir}}{\partial x^{ks}}$ instead of $\frac{\partial G^{ir}}{\partial x^{ks}}(x)$
and $\pi^{ir}$ in place of $\pi^{ir}\!(x^i)$.
We also use the Kronecker delta  $\delta_{ij}$ 
which is equal to 1 if $i=j$ and 0 otherwise.

\vspace{1ex}
\begin{lemma} Let $x\in S$ be a rest point for the dynamics \eqref{ad}. 
For each $r\in R$ and $i,k\in I$ let us denote $u^r_{\mbox{-}ik}=\sum_{j\neq i,k}Y^{jr}$ and define
$$\Lambda_{ik}^r=\left\{\begin{array}{cl}
\mathbb{E}_\pi[C^r_{2+u_{\mbox{-}ik}^r}\!- C^r_{1+u_{\mbox{-}ik}^r}]&\mbox{if }k\neq i,\\
0&\mbox{if }k=i.
\end{array}\right.$$
Then the 
Jacobian $J(x)$ has entries
\begin{equation}
\frac{\partial G^{ir}}{\partial x^{ks}}=\pi^{ir} \big( \beta_k\pi_k^r(\pi_k^s-\delta_{rs})\Lambda_{ik}^r - \delta_{ik}\delta_{rs}\big).\label{partial113}
\end{equation}\end{lemma}
\begin{proof}
Since a rest point $x\in S$ satisfies $C^{ir}\!(x^{\mbox{-}i})=x^{ir}$ it follows that
\begin{eqnarray}\nonumber
\frac{\partial G^{ir}}{\partial x^{ks}}&=&\frac{\partial \pi^{ir}}{\partial x^{ks}} \; (C^{ir}\!(x^{\mbox{-}i}) - x^{ir}) + \pi^{ir} \; \frac{\partial}{\partial x^{ks}} \; (C^{ir}\!(x^{\mbox{-}i}) - x^{ir})\\
&=&\pi^{ir} ( \mbox{$\frac{\partial C^{ir}}{\partial x^{ks}}$} - \delta_{ik}\delta_{rs}).\label{partial111}
\end{eqnarray}
In order to compute $\frac{\partial C^{ir}}{\partial x^{ks}}$ we note that $C^{ir}\!(x^{\mbox{-}i})$ does not depend on 
$x^i$ so that for $k=i$ these derivatives are all zero. On the other hand, for $k\neq i$ the function $C^{ir}\!(x^{\mbox{-}i})$ depends 
on $x^{ks}$ only as an affine function of $\pi_k^r=\pi_k^r(x^k)$. More precisely, conditioning on $Y^{kr}$ in the expression 
\eqref{expected}  we get
\begin{eqnarray*}\label{expected2}
C^{ir}\!(x^{\mbox{-}i})&=&\pi_k^r\;\mathbb{E}_\pi[C^r_{2+u_{\mbox{-}ik}^r}] +(1\!-\!\pi_k^r)\;\mathbb{E}_\pi[C^r_{1+u_{\mbox{-}ik}^r}]\\
&=&\mathbb{E}_\pi[C^r_{1+u_{\mbox{-}ik}^r}]+\pi_k^r\;\mathbb{E}_\pi[C^r_{2+u_{\mbox{-}ik}^r}\!\!-C^r_{1+u_{\mbox{-}ik}^r}]
\end{eqnarray*}
from which it follows that
\begin{equation}\label{partial112}
\frac{\partial C^{ir}}{\partial x^{ks}}=\Lambda_{ik}^r\,\frac{\partial \pi_k^r}{\partial x^{ks}}.
\end{equation}
Now, for the Logit probabilities \eqref{adp} we have
$$\frac{\partial {\pi}_k^r}{\partial x^{ks}} =\left\{\begin{array}{cl}
- {\beta}_k {\pi}_k^r(1\! -\! {\pi}_k^r) & \mbox{if }s=r\\
{\beta}_k {\pi}_k^r\pi_k^s& \mbox{if }s\neq r
\end{array}\right.
$$
which combined with \eqref{partial112} and replaced into \eqref{partial111} yields \eqref{partial113}.
\end{proof}

Using \eqref{partial113} we have more explicitly
\begin{eqnarray} 
&\mbox{for $k=i$}\qquad&\frac{\partial G^{ir}}{\partial x^{is}} =\left\{\begin{array}{cl} -\pi^{ir}&\mbox{if }s=r\\
0&\mbox{if } s\neq r\end{array} \right. \label{gnd}\\
&\mbox{for $k\neq i$}\qquad&
\frac{\partial G^{ir}}{\partial x^{ks}} = 
\left\{\begin{array}{cl} - {\beta}_k \pi^{ir} {\pi}_k^r (1\!-\!{\pi}_k^r){\Lambda}_{ik}^r &\mbox{if }s=r \\
{\beta}_k \pi^{ir} {\pi}_k^r {\pi}_k^s\, {\Lambda}_{ik}^r&\mbox{if }s\neq r.
\end{array}\right.\label{girks}
\end{eqnarray}
Note that since $C^r_u$ increases with $u$ we have ${\Lambda}_{ik}^r\geq 0$ and  
the derivatives have a definite sign: $\frac{\partial G^{ir}}{\partial x^{ks}}\leq 0$ if $s=r$
and $\frac{\partial G^{ir}}{\partial x^{ks}}\geq 0$ for $s\neq r$.

It follows that the Jacobian can be organized 
into $N^2$ blocks as
\begin{displaymath}
J = \left(
\begin{array}{ccc}
J^{11} & \cdots & J^{1N} \\
\vdots & \ddots & \vdots \\
J^{N1} & \cdots & J^{NN} \\ 
\end{array} 
\right), \label{je}
\end{displaymath}
where the blocks $J^{ik}$ are of order $M \times M$ and are given by
\begin{eqnarray}
J^{ii} & = & -\mbox{diag}\big( (\pi^{ir})_{r \in R}\big), \label{jd} \\
J^{ik} & = & \big({\beta}_k \pi^{ir}{\pi}_k^r({\pi}_k^s-{\delta}_{sr}){\Lambda}_{ik}^r\big)_{r,s \in R}. \label{jod}
\end{eqnarray}
Each block $J^{ik}$ has negative diagonal and positive off-diagonal entries, and
$\mbox{trace}(J^{ii})=-\sum_{r\in R}\pi^{ir}=-1$ so that
\begin{equation}
\mbox{trace}(J) = \sum_{i = 1}^N \mbox{trace}(J^{ii}) = - N. \label{tr}
\end{equation}
In order to determine stability of an equilibrium $x\in S$, we must verify that all the eigenvalues of the Jacobian $J$ 
have negative real part.  The eigenvalues are the roots of the characteristic polynomial  $p(\lambda)=\det(\lambda I- J)$,
which can be expanded as
\begin{eqnarray}\nonumber
p(\lambda)=a_0 {\lambda}^{n} + a_1 {\lambda}^{n-1} + \cdots + a_{n-1} \lambda + a_{n} \label{cp}
\end{eqnarray}
with $a_0 = 1, a_1 = - \mbox{trace}(J) = N$ and $a_{n} = (-1)^{n} \mbox{det}(J)$. 

The Routh-Hurwitz stability criterion states that all the roots of the real polynomial $p(\lambda)$ (with $a_0 > 0$) have negative real part if and only if the leading minors $\{\Delta_k(p)\}_{k= 1}^n$ of the Hurwitz matrix below are positive
{\scriptsize
\begin{equation}
H(p) = \left(
\begin{array}{ccccccccc}
a_1 & a_3 & a_5 & \cdots &\cdots&\cdots & 0& 0&0 \\
a_0 & a_2 & a_4 & &  & & \vdots &\vdots & \vdots\\
0 & a_1 & a_3 &  & & &  \vdots&\vdots & \vdots\\
\vdots & a_0 & a_2 & \ddots &  & &0 & \vdots& \vdots\\
\vdots & 0 & a_1 &  & \ddots & & a_n& \vdots& \vdots\\
\vdots & \vdots & a_0 &  &  & \ddots & a_{n-1}&0 & \vdots\\
\vdots & \vdots & 0&  &  &   & a_{n-2}&a_{n} & \vdots\\
\vdots & \vdots & \vdots & &  & & a_{n-3}  &a_{n-1}   &0\\
0 & 0&0 & \cdots& \cdots& \cdots& a_{n-4}&a_{n-2}&a_n
\end{array} 
\right). \label{hm}
\end{equation}
}

This matrix has diagonal entries $a_1,\ldots,a_n$ and  each column 
contains the coefficients $a_i$ in decreasing order with the convention 
$a_i=0$ for $i<0$ and $i>n$.
Recall that a leading or principal minor is the determinant of the square sub-matrix obtained by keeping only the first $k$ rows and columns.
In our case $a_0 = 1$ and $a_1 = N$,  and the first four leading minor conditions are
\begin{eqnarray}
0<{\Delta}_1(p) & = & N, \label{m1} \\
0<{\Delta}_2(p) & = & N a_2 - a_3, \label{m2} \\
0<{\Delta}_3(p) & = & (N a_2 -a_3)a_3 + N a_5 - N^2 a_4, \label{m3} \\
0<{\Delta}_4(p) & = & N (a_2 a_3 a_4 + N a_2 a_6 + a_4 a_5 - a_3 a_6 - a_5 a_2^2 - N a_4^2) \nonumber \\
&  &{}-(a_4 a_3^2 + a_5^2 + N a_2 a_7 - a_3 a_7 - a_2 a_3 a_5 - N a_4 a_5). \label{m4}
\end{eqnarray}
Note that when $n=4$ we have $a_5=a_6=a_7=0$ so that the expression \eqref{m4} factorizes
as $a_4\Delta_3(p)$ and the condition $\Delta_4(p)>0$ reduces to $a_4>0$.

\vspace{1ex}
\noindent{\sc Remark.} The expression of  $\Delta_k(p)$ 
becomes increasingly complex as $k$ grows. It is worth noting that, according to
\cite[Dimitrov and Pe\~na]{dimitrov}, a sufficient 
condition for stability  is that $a_k>0$ and
$a_k a_{k+1} \ge \bar z\, a_{k-1} a_{k+2}$ for all $k \!=\! 1,\ldots,n\!-\!2$
where  $\bar z\approx 4.08$ is the unique real root of $z^3\! -\! 5 z^2\! +\! 4z\! -\!1\!=\!0$.
\vspace{1ex}

The Routh-Hurwitz criterion provides necessary and sufficient conditions for the asymptotic stability 
of the rest points of a system of nonlinear differential equations, and it can be applied 
even without  solving explicitly the steady state equations that arise by setting the time 
derivatives to zero. This tool is frequently utilized in the study of dynamical systems, especially in mathematical 
biology and control system theory, but it has also appeared occasionally in the context of games.  
         It was mentioned by Hofbauer and Sigmund in \cite[Example 15.6.10]{hofbauer} in connection with the 
         stability of linear Lotka-Volterra systems and replicator dynamics in games, and it was also suggested by Cressman in 
         \cite[Section 3]{cressman} as a plausible tool to characterize  
         the evolutionary stable strategies in multi-species systems.
         Unlike these examples, which describe the aggregate evolution of the frequency of strategies in a population 
         game, in this paper we apply the criterion to analyze the adaptive behavior of 
         individual players in a finite congestion game.

 The Routh-Hurwitz criterion is more general than other stability criteria. For instance, the Nyquist  criterion in control theory 
 is restricted to linear time-invariant systems, while its generalization known as the Circle criterion applies only for suitably small  
 nonlinear perturbations. Other criteria, such as Lyapunov stability, examine how 
 quickly the convergence of trajectories near a stable equilibria occurs, by constructing a potential function. Like 
 Routh-Hurwitz, the Lyapunov criterion does not need to find  the steady states of the system to determine their 
 stability, but it relies on the ability to find an appropriate potential. In this sense
 Routh-Hurwitz is simpler as it only requires  to examine the signs of certain sub-determinants of the Jacobian matrix of the system.
On the down side, as the dimension of the system increases the Routh-Hurwitz criterion must deal with larger sub-determinants 
which may become cumbersome. However, for moderate dimensions it can be used to fully characterize
the stability of rest points, as in the case of the learning dynamics in the $2\times 2$ routing game considered in the next section.

\section{Symmetric $2\times 2$ traffic game}\label{sec4}
In this section we consider the special case in which we have only two routes $r=a,b$
and two drivers $i=1,2$. 
In this case the system \eqref{eq} becomes
\begin{equation}\label{eq22}
\begin{array}{lclcl}\hline
x^{1a}&=&C_1^a\pi^{2b}+C_2^a\pi^{2a}&=&C_1^a+(C_2^a\!-\!C_1^a)\,\pi^{2a}\\
x^{1b}&=&C_1^b\pi^{2a}+C_2^b\pi^{2b}&=&C_2^b+(C_1^b\!-\!C_2^b)\,\pi^{2a}\\ \hline
x^{2a}&=&C_1^a\pi^{1b}+C_2^a\pi^{1a}&=&C_1^a+(C_2^a\!-\!C_1^a)\,\pi^{1a}\\
x^{2b}&=&C_1^b\pi^{1a}+C_2^b\pi^{1b}&=&C_2^b+(C_1^b\!-\!C_2^b)\,\pi^{1a}\\ \hline
\end{array}
\end{equation}

\subsection{Stability of equilibria}\label{sec41}
In order to give a simpler expression for the characteristic polynomial $p(\lambda)$ 
it is convenient to introduce the constants
\begin{equation}\label{constants}
\begin{array}{lcl}
{\delta}^a&=& C_2^a\! -\! C_1^a\\
{\delta}^b&=& C_2^b\! -\! C_1^b\\
\delta&=&\delta^a+\delta^b\\
\nu&=&{\beta_1\beta_2}.\end{array}
\end{equation}

\vspace{1ex}
\begin{lemma} For $i=1,2$ let $z_i=\pi^{ia}\pi^{ib}$ and $y_i=\pi^{ia}{\delta}^a  + \pi^{ib}{\delta}^b $, and 
denote  
$$\begin{array}{ccl}
Z &= &z_1z_2\\
S &= &z_1 + z_2
\end{array}
\qquad ;\qquad 
\begin{array}{ccl}
Y &= &y_1y_2\\
P &=& z_1 y_2 +  z_2y_1.
\end{array}
$$
Then, the characteristic polynomial of the Jacobian \eqref{jac22} is given by
\begin{eqnarray}
p(\lambda) &=& {\lambda}^4 + 2 {\lambda}^3 + a_2 {\lambda}^2 + a_3 \lambda + a_4\label{cp22}\\[1ex]
a_2 & = & 1+S - \nu ZY \label{a222} \\
a_3 & = & S - \nu\delta Z P \label{a322} \\
a_4 & = & Z (1-\nu\delta^2 Z). \label{a422}
\end{eqnarray}
\end{lemma}
\begin{proof}
Using \eqref{jd}-\eqref{jod} the Jacobian $J$ can be expressed as
\begin{equation}\label{jac22}
J=\left(\begin{array}{cc|cc}
-\pi^{1a} & 0 &-\beta_2z_2\pi^{1a}\delta^a&\hspace{0.5ex}\beta_2z_2\pi^{1a}\delta^a \\ 
0& -\pi^{1b}& \hspace{1.2ex}\beta_2z_2\pi^{1b}\delta^b&\hspace{-1.2ex}-\beta_2z_2\pi^{1b}\delta^b\\ \hline
\hspace{-1.2ex}-\beta_1z_1\pi^{2a}\delta^a&\hspace{0.5ex}\beta_1z_1\pi^{2a}\delta^a&-\pi^{2a} & 0\\
\beta_1z_1\pi^{2b}\delta^b&\hspace{-1.2ex}-\beta_1z_1\pi^{2b}\delta^b& 0& -\pi^{2b} 
\end{array}
\right).
\end{equation}
From this expression the computation of the characteristic polynomial $p(\lambda)$  is a  routine exercise 
and is omitted.
\end{proof}
\drop{
{\color{cyan}
Observe that once the rest point $x$ and the constant $\vec \beta = ({\beta}_1,{\beta}_2)$ are fixed, quantities $Z = Z(x,\vec \beta), S = S(x,\vec \beta), Y = Y(x,\vec \beta), P = P(x,\vec \beta), \nu = \nu(\vec \beta)$ are fixed. The dynamical system analysis that follows will utilize $\delta$ as a bifurcation parameter. The removal of the parameter $\beta$ from the dynamical system analysis obeys to the simplicity of the interpretation of the results that follow and not by the impossibility of the analysis.
}
}
In the sequel we will see that the sign of $a_4$ plays a crucial role in determining 
the stability of the rest point. To this end we will first establish some useful 
inequalities for the coefficients of the characteristic poynomial.

\vspace{1ex}
\begin{lemma} \label{L3} The following inequalities hold
\begin{eqnarray}
2a_2-a_3&\geq&S+2a_4/Z \label{euno}\\
a_3&\geq&Sa_4/Z\label{edos}
\end{eqnarray}
with strict inequality unless $\delta=0$.
Moreover, if $a_4\geq 0$  we have $2a_2-a_3>0$, $a_2>0$, and $a_3>0$.
\end{lemma}

\begin{proof} We note that $y_i=\pi^{ia}\delta^a +\pi^{ib}\delta^b\leq\delta$
so that $Y\leq\delta^2$ and $P\leq\delta S$. Moreover, these inequalities are strict unless $\delta=0$. 
From \eqref{a222}-\eqref{a322} we get
\begin{eqnarray*}
2a_2-a_3&=&S+\nu\delta ZP+2(1-\nu ZY)\\
&\geq&S+\nu\delta ZP+2(1-\nu\delta^2Z)
\end{eqnarray*}
and \eqref{euno} follows since $\nu\delta ZP\geq 0$ with strict inequality unless $\delta=0$. 
Also
$$a_3 = S - \nu\delta Z P\geq S-\nu\delta^2 ZS=Sa_4/Z$$
again with strict inequality unless $\delta=0$.
 
Let us suppose now that $a_4\geq 0$. Since $S>0$, from \eqref{euno} we get $2a_2-a_3>0$ 
while \eqref{edos} gives $a_3\geq 0$ which together imply $a_2>0$. It remains to show that
$a_3>0$. From \eqref{edos} this is clear if $a_4> 0$. Otherwise, when $a_4=0$ we have $\nu\delta^2Z=1$
so that $\delta>0$ and the inequality \eqref{edos} is strict.
\end{proof}

Using the previous Lemma we can establish that the rest point is stable if and only if $a_4 > 0$. The proof
exploits the fact that $X$ and $Y$ are strictly positive while
$Z$ and $W$ are non-negative.

\vspace{1ex}
\begin{theorem}\label{TT1}
Let $x\in S$ be a rest point for 
a $2\times 2$  game with corresponding Logit probabilities  $\pi^{ir}$.
Then $x$ is a stable point for  \eqref{ad} if and only if $a_4>0$, that is to say
\begin{equation}\label{stab}
\nu\delta^2\,\pi^{1a}\pi^{1b}\pi^{2a}\pi^{2b}<1.
\end{equation}
\end{theorem}
\begin{proof}
As noted in the previous section, the Routh-Hurwitz conditions 
\eqref{m1}-\eqref{m4} are equivalent to
\begin{eqnarray}
2 a_2 - a_3 & > & 0, \label{rh222} \\
(2 a_2 - a_3) a_3 - 4 a_4 & > & 0, \label{rh322} \\
a_4 & > & 0. \label{rh422}
\end{eqnarray}
From \eqref{a422} we see that \eqref{rh422} is equivalent to 
$\nu\delta^2Z<1$ which is exactly \eqref{stab}. 
Now, Lemma \ref{L3} shows that when $a_4>0$ we have $2a_2-a_3>S$ so that
\eqref{rh422} implies  \eqref{rh222}. Moreover, combining this inequality with \eqref{edos} we get
$$(2 a_2 - a_3)a_3  - 4 a_4>S^2a_4/Z-4 a_4=(z_1-z_2)^2a_4/Z \geq 0$$
showing that  \eqref{rh422}  also implies \eqref{rh322}.
Hence, \eqref{rh222} and \eqref{rh322} are superfluous and the Routh-Hurwitz stability 
stability conditions are reduced to \eqref{stab}.
\end{proof}

Since $\pi^{ia}+\pi^{ib}=1$,  the stability condition \eqref{stab} can be written as
$$\pi^{1a}(1\!-\!\pi^{1a})\,\pi^{2a}(1\!-\!\pi^{2a})< (\nu\delta^2)^{-1}$$
%which can also be written as
%$$\mbox{$\left[ \left( \pi^{1a}\!-\!\frac{1}{2} \right)^2-\frac{1}{4} \right] \left[ \left( \pi^{2a}\!-\!\frac{1}{2} \right)^2-\frac{1}{4} \right] < \delta^{-2}$}.$$
which are level sets of the function $f(x,y)=x(1-x)y(1-y)$.
Figure \ref{fi:e0} plots this funtion and the contours of its level sets.
The region of stability is the sector comprised between the corresponding contour and the boundary of the  rectangle
$[0,1]^2$. Since the maximum of $f(x,y)$ is attained at $(\frac{1}{2},\frac{1}{2})$ with  value $\frac{1}{16}$, it follows that
when $\nu\delta^2<16$ the stability region is the full rectangle $[0,1]^2$ and all equilibria are stable.
\begin{figure}[!tbh]
\centering
\includegraphics[scale=0.38]{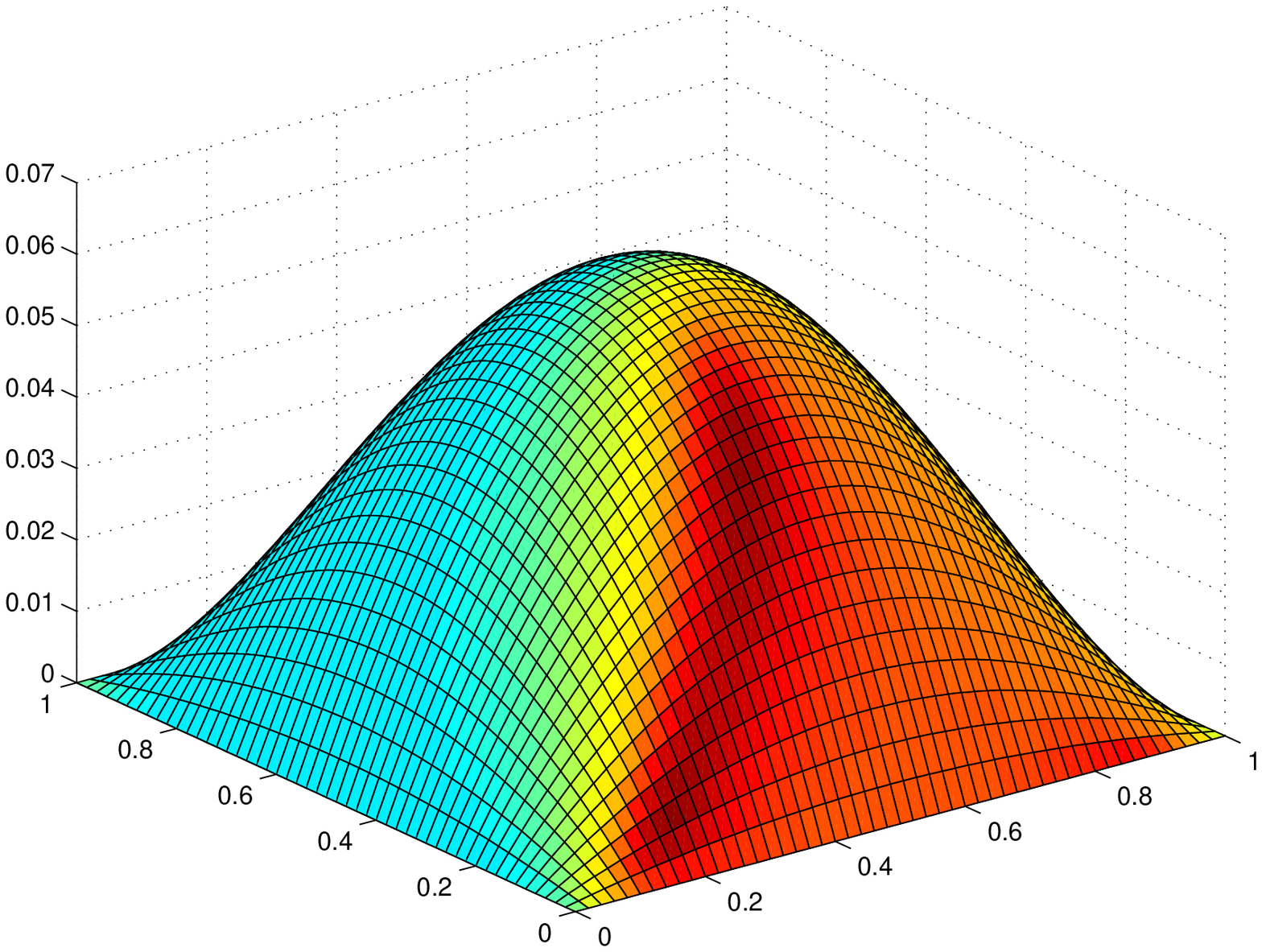}\hspace{3ex}
\includegraphics[scale=0.40]{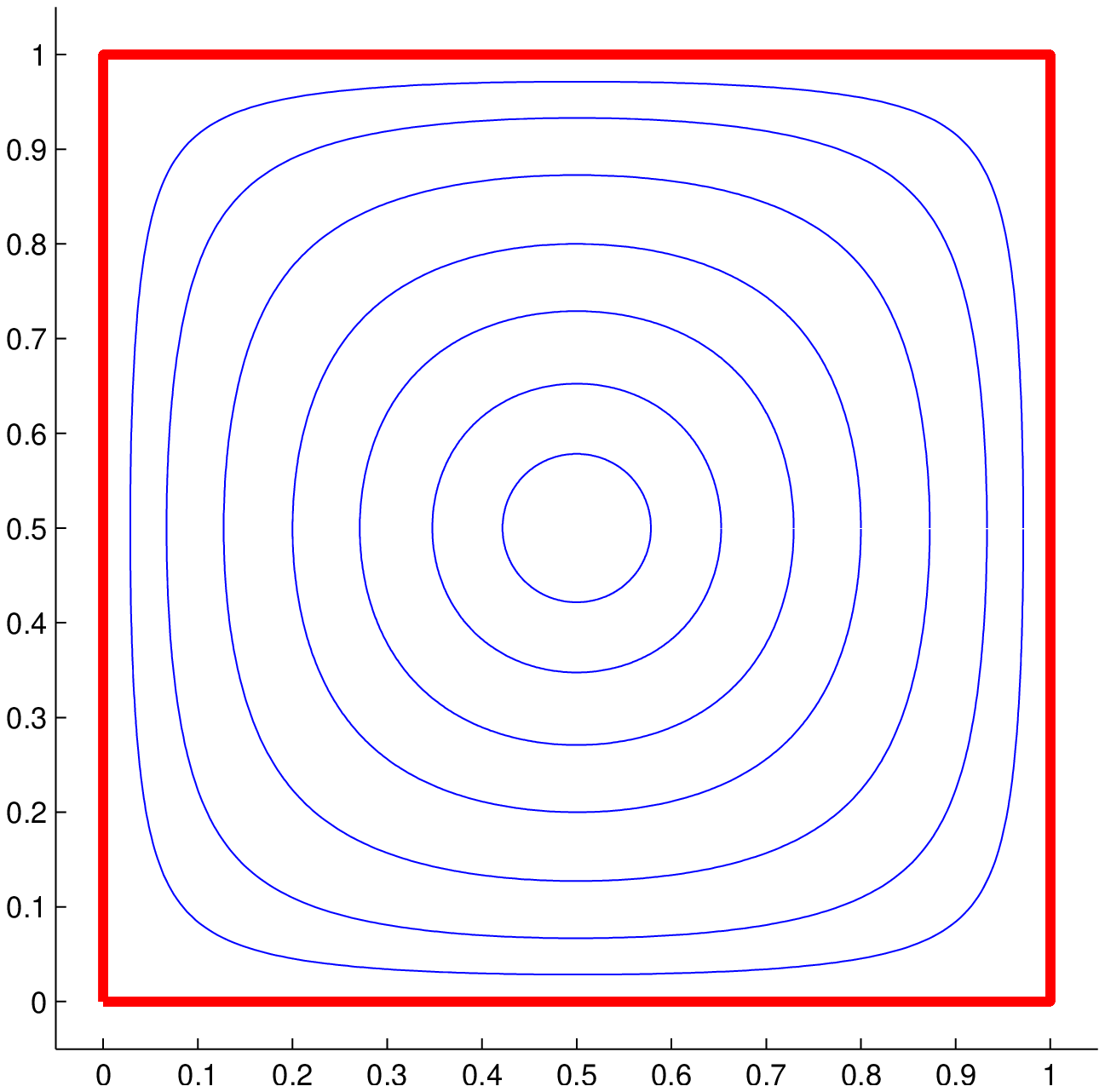}
\caption{Stability region in the $2 \times 2$  traffic game} 
\label{fi:e0}
\end{figure}

When $a_4 = 0$ the characteristic polynomial can be factored as $p(\lambda)=\lambda q(\lambda)$ with $q(\lambda)={\lambda}^3 + 2 {\lambda}^2 + a_2 \lambda + a_3$. So $p(\lambda)$ has a null eigenvalue. 
Applying the Routh-Hurwitz criterion to $q(\lambda)$
and using the inequalities $2a_2-a_3>0$ and $a_3>0$ in Lemma \ref{L3}, it follows that the 
other three eigenvalues of $p(\lambda)$ have negative real part. Therefore, the equilibrium is still stable. 

When $a_4$ becomes negative we have $p(0)=a_4<0$ and since $p(\pm\infty)=\infty$ it follows that
there are at  least one negative and one positive eigenvalue so that the rest point is an unstable
saddle point for the dynamics. In fact, by a continuity argument, when $a_4$ is slightly negative, $a_2$ and $a_3$ 
are still positive and $p(\lambda)$ will have exactly one positive root and the  other three roots have negative real part. 
Hence, for $a_4$ slightly positive, the rest point is unstable due to a crossing of the imaginary axis of a simple real negative root.
This  does not discard the presence of a Hopf bifurcation or other more complex dynamical behavior. The non-existence of Hopf 
bifurcation will be established later by a $K_s$-monotonicity argument.

\drop{
{\color{red} RC: No entiendo el comentario que sigue. Los equilibrios y la funci\'on $Z$ depende de todos los par\'ametros 
$\delta^a,\delta^b,\beta_1,\beta_2$. Por consiguiente no me queda claro que el an\'alisis se pueda reducir a un \'unico par\'ametro 
de bifurcaci\'on $\delta$}

{\color{cyan}
As mentioned above, given the rest point $x$ and the constant $\vec \beta$, the Routh-Hurwitz positivity condition on $a_4$ translates to the condition on the bifurcation parameter $\delta$
\begin{equation}
\delta < {\delta}_0 = (\nu(\vec \beta) Z(x,\vec \beta))^{-1/2}, \label{delta_cond}
\end{equation}
where ${\delta}_0$ is a bifurcation value for the bifurcation parameter $\delta$. 
}
}
Although we have established the existence of rest points, we have not yet determined how many of these equilibria exist. 
Using a fixed-point argument based on a one dimensional dynamical map associated with the dynamical system, we will 
see that depending on the values of the parameters there may be multiple equilibria. 

\drop{
It will turn out that when this condition occurs, the stable rest point under consideration 
will turn unstable and two new stable equilibria will show up, pointing out the rising of a fold bifurcation in $\delta$ when it crosses 
the value ${\delta}_0$ form left to right.
}

\subsection{Counting stable and unstable equilibria}\label{sec42}

We claim that in the $2\times 2$ case there are either one, two or three rest points. To see this  we reduce \eqref{eq22} 
to a fixed point equation in dimension one. 
Namely, we note that a solution $x$ of \eqref{eq22} is fully determined once we know $\pi^{1a}$ and $\pi^{2a}$. 
Moreover, denoting $w^i=\beta_i(x^{ia}\!-\!x^{ib})$   these Logit probabilities are 
$\pi^{ia}=\rho(w^i)$ where $\rho(w)=1/[1\!+\!\exp(w)]$ so that \eqref{eq22} can be reduced to a $2\times 2$ 
system in the unknowns $w^1$ and $w^2$.
Indeed, substracting the equations in each block of \eqref{eq22} and 
setting $\varphi(w)=\kappa+\delta\rho(w)$ with $\kappa=C_1^a-C_2^b$,
we get 
\begin{equation}\label{eq22prime}
\begin{array}{l}
w^1=\beta_1\,\varphi(w^2)\\
w^2=\beta_2\,\varphi(w^1).
\end{array}
\end{equation}  
Hence $w^1$ is a fixed point of  the scalar function $\psi(w)\triangleq\beta_1\varphi(\beta_2\varphi(w))$,
 which yields a solution of \eqref{eq22prime} with $w^2=\beta_2\,\varphi(w^1)$.
This establishes a one to one correspondence between the solutions of \eqref{eq22} and the fixed points of 
 $\psi$. Moreover, using \eqref{psip} and noting that
$\varphi'(w)\!=\!\delta\rho(w)(1\!-\!\rho(w))$ we get
\begin{eqnarray}\nonumber
\psi'(w^1)&=&\nu\varphi'(w^2)\varphi'(w^1)\\
&=&\nu\delta^2\pi^{1a}(1\!-\!\pi^{1a})\pi^{2a}(1\!-\!\pi^{2a})\nonumber\\
&=&\nu\delta^2\pi^{1a}\pi^{1b}\pi^{2a}\pi^{2b}\label{eqeq}
\end{eqnarray}
so that the rest point $x$ is stable if and only if $\psi'(w^1)<1$. 
 
\vspace{1ex}
\begin{theorem} \label{TT2}
The function $\psi$ has either one, two or three fixed points, and 
exactly one of the following mutually exclusive situations occurs.
\begin{itemize}
\item[(a)] $\psi$ has a unique fixed point $\bar w$ with $\psi'(\bar w)\leq 1$.
The corresponding $\bar x$ is the unique equilibrium of \eqref{ad} and it is stable iff $\psi'(\bar w)< 1$.
\item[(b)] $\psi$ has two fixed points $w_-<w_+$ with either  $\psi'(w_-)<1=\psi'(w_+)$ or $\psi'(w_+)<1=\psi'(w_-)$.
The corresponding points $x_-$ and $x_+$ are the only rest points of \eqref{ad}. One of them is stable and the other is unstable.
\item[(c)] $\psi$ has three fixed points $w_-\!<\!\bar w\!<\!w_+$ with $\psi'(w_-)<1$, $\psi'(w_+)<1$, and $\psi'(\bar w)>1$. 
The corresponding points $x_-$,  $x_+$ and  $\bar x$ are the only rest points of \eqref{ad}, with $x_+$ and $x_-$  
stable and  $\bar x$ unstable.
\end{itemize}
\end{theorem}
\begin{proof}
Let us first observe that when $\delta=0$ the function $\psi$ is constant and has a unique fixed point $\bar w=\beta_1\kappa$
which falls in the situation {\em (a)}. Let us then consider the case $\delta>0$ so that  $\varphi$ is  strictly decreasing.
In this case $\psi$  is strictly increasing and bounded so that it must cross the identity at least once.
On the other hand, a direct calculation yields
\begin{eqnarray}\label{psip}
\psi'(w)\!\!&=&\nu\varphi'(\beta_2\varphi(w))\varphi'(w)\\[1ex]
\psi''(w)\!\!&=&\nu\beta_2\varphi''({\beta}_2 \varphi(w))\varphi'(w)^2+\nu\varphi'({\beta}_2 \varphi(w))\varphi''(w)\nonumber\\
&=&{}\!\!\frac{\nu\delta^2e^{\beta_2\varphi(w)}e^{2w}}{(1\!+\!e^{\beta_2\varphi(w)})^2(1\!+\!e^w)^4}\left[\beta_2\delta-\frac{2\beta_2\delta}{1\!+\!e^{\beta_2\varphi(w)}}-2\sinh(w)\right]\!.
\end{eqnarray}
Recall that $\sinh x = \frac{1}{2}(e^x - e^{-x}), \; x \in \mathbb R$.
The expression in the last square bracket is strictly decreasing 
in $w$ so that $\psi''$ has exactly one zero $w_0$ and therefore $\psi$ is strictly convex on  $(-\infty,w_0]$
and strictly concave on $[w_0,\infty)$. It follows that $\psi$ can cross the identity at most three times  (see Fig. \ref{fi:fp133})
and one of the mutually exclusive situations {\em (a)}-{\em (c)} must occur.
\end{proof}

Figure \ref{fi:fp133} below illustrates  case {\em (c)} where $\psi$ has 3 fixed points with $\psi'(\bar w)>1$.
Depending on the values of the parameters $\kappa,\delta$ and $\nu$, the graph of $\psi$ might shift to the right
so that eventually $\bar w$ and $w_+$ may collapse into a single fixed point producing situation {\em (b)}, after which 
this double fixed point disappears leading to case {\em (a)} with $w_-$ as the only fixed point. Symetrically, if the graph
of $\psi$ is shifted to the left, $\bar w$ may collapse with $w_-$ producing situation {\em (b)}, after which we fall
into case {\em (a)} with unique fixed point $w_+$. Note that the case {\em (a)} with $\psi'(\bar w)=1$ can only
occur when $\psi$ has an inflection at $\bar w$.
\begin{figure}[!tbh]
\centering
\includegraphics[scale=0.5]{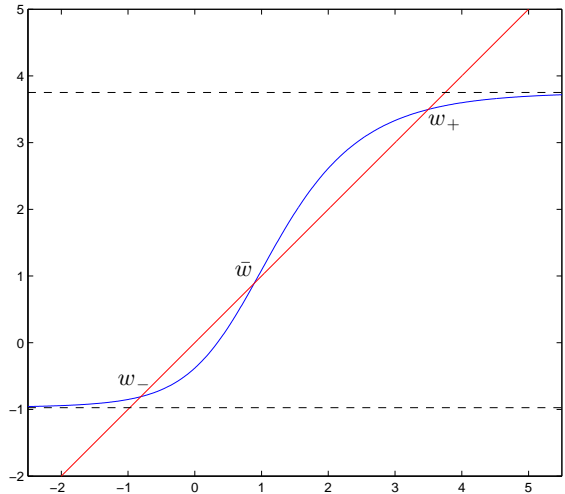}
\caption{Three fixed points with $\psi'(\bar w)>1$.} 
\label{fi:fp133}
\end{figure}

Since $\psi$ is a strictly increasing function in one dimension, from dynamical system theory for maps \cite{guckenheimer}, it follows that for generic cases in Theorem \ref{TT2}, the fixed point is an orientation preserving sink when ${\psi}^{'} < 1$ and an orientation reversing source if ${\psi}^{'} > 1$. In the latter case, two additional fixed points appear ($w_{-}$ and $w_{+}$) which correspond to two stable steady states of the ODE system, indicating a topological change in the structure of the mean field dynamics which exhibits a bifurcation of the fold kind \cite{kuznetsov}. 
\drop{
These analysis complements and is consistent with the results obtained by analyzing the stability of rest points via the Routh-Hurwitz stability criterion, as indicated by the introduction of the critical value ${\delta}_0$ (see (\ref{delta_cond})) of the bifurcation parameter $\delta$.

{\color{red} RC: La misma duda que antes  respecto de esta \'ultima frase --- No me queda claro que baste considerar a $\delta$ como par\'ametro de bifurcaci\'on dado que los equilibrios dependen de todos los par\'ametros. Yo omitir\'ia el \'ultimo comentario ``These analysis complements...'' }
}

In case of multiplicity the rest points satisfy an order relation. Namely, let  $s=(1,-1,-1,1)$ and
consider the partial order $\leq_s$  induced by  the orthant  
$$K_s = \{ x \in {\mathbb R}^4: s_i x_i \ge 0\mbox{ for all } i=1,\ldots, 4 \},$$
that is to say, $x\leq_s y$ if and only if $s_i(y_i-x_i)\geq 0$ for $i=1,\ldots,4$.

\vspace{1ex}
\begin{proposition}\label{order} 
Consider case (c) in Theorem {\em \ref{TT2}} where $\psi$ has three fixed points $w_-\!<\!\bar w\!<\!w_+$. 
Then the corresponding rest points $x_-$,  $x_+$ and  $\bar x$ are ordered as $x_-\leq_s\bar x\leq_s x_+$.
\end{proposition}
\begin{proof}
Since $w^1_-=w_-<w_+=w^1_+$, and since $\rho(\cdot)$ is decreasing,  the corresponding 
Logit probabilities  satisfy the inequality 
$$\pi^{1a}_-=\rho(w^1_-)>\rho(w^1_+)=\pi^{1a}_+$$
so that
the third and fourth equations in \eqref{eq22} yield $x^{2a}_-\geq x^{2a}_+$ and 
$x^{2b}_-\leq x^{2b}_+$.  Now, since $\varphi(\cdot)$ is decreasing it follows that 
$w^2_-=\beta_2\varphi(w^1_-)>\beta_2\varphi(w^1_+)=w^2_+$
and similarly as before we get 
$$\pi^{2a}_-=\rho(w^2_-)<\rho(w^2_+)=\pi^{2a}_+$$
so that the first and second equations in \eqref{eq22} yield  $x^{1a}_-\leq x^{1a}_+$ and 
$x^{1b}_-\geq x^{1b}_+$.
Hence $x_-\leq_s x_+$. Analogous arguments yield
$x_-\leq \bar x$ and $\bar x\leq_s x_+$.
\end{proof}

\subsection{$K_s$-monotonicity and asymptotics of the dynamics}\label{sec43}

The symmetric sign distribution of the Jacobian coefficients in (\ref{jac22}) motivates the use of results of monotone or order-preserving dynamical systems theory \cite{smith}. 
Such dynamical systems naturally occur in biology, chemistry, physics and economics, where the notion of cooperative vector field shows up. For those vector fields, a natural order relationship 
among trajectories rises which severely restricts the long-term behavior of the positive semi-flows. 
In the field of game theory, cooperative dynamics have been used to study the asymptotics of 
 best response dynamics as well as fictitious play for supermodular games (see {\em e.g.} \cite{beggs,benaim_faure,benaim_hirsch,borkar,hofbauer_sandholm}).
Here we use the results for $K_s$-monotonic dynamical systems to describe the global behavior of the learning dynamics in the $2 \times 2$ 
traffic game. Note however that, in contrast with the previously cited applications in games, the nature of congestion games leads to competitive rather than cooperative dynamics.

Let $s=(1,-1,-1,1)$ as in the previous section. The dynamics \eqref{ad} turn out to be  $K_s$-monotone, 
which gives a more complete picture of its asymptotic behavior. 
Recall \cite[Smith]{smith} that a flow ${\phi}_t(x_0)$ generated by a system of ordinary differential equations
$$\left\{\begin{array}{l}
\dot x(t) = f(x(t)),\\
x(0)=x_0
\end{array}\right.
$$
is called $K_s$-monotone if it  
preserves the partial order $\leq_s$, namely, 
for all $x_0,y_0\in U$ such that $x_0 \; {\le}_s \; y_0$ it holds that 
${\phi}_t(x_0) \; {\le}_s \; {\phi}_t(y_0)$ for all times $t\geq 0$ at which both solutions are defined.

We will use Lemma 2.1 in \cite[Smith]{smith} which gives necessary and sufficient conditions for the flux of 
an ODE system to be $K_s$-monotone in terms of the signs of the entries of the Jacobian matrix of the system.
%In addition \cite[Lemma 2.2]{smith} establishes conditions for a monotone system to be irreducible.

\vspace{1ex}
\begin{proposition}\label{mono}
Consider a $2 \times 2$ traffic game.
Then,
\begin{itemize}
\item[a)] The mean field system \eqref{ad} is $K_s$-monotone for $s=(1,-1,-1,1)$.
\item[b)] If there exists  some $T>0$ such that $\phi_T(x_0)\!<_s\!x_0$ or   
$x_0\!<_s\!\phi_T(x_0)$, then the orbit $\phi_t(x_0)$ converges 
to a rest point as $t\to \infty$.
\end{itemize}
\end{proposition}

\begin{proof}
According to \cite[Lemma 2.1]{smith},  a system is $K_s$-monotone if and only if 
for all $x \in U$ the matrix $P_s J(x) P_s$ has non-negative off-diagonal 
elements, where $J(x)=Df(x)$ and $P_s = \mbox{diag }(s_1,\ldots,s_n)$.
Since the Jacobian  of  \eqref{ad}  is given 
by \eqref{jac22}, the off-diagonal terms $(P_s J(x) P_s)_{ij} = s_i s_j J_{ij}$
are non-negative if and only if
$$s_1 s_3 < 0~;~s_1 s_4 > 0~;~s_2 s_3 > 0~;~s_2 s_4 < 0. 
$$
This holds for $s=(1,-1,-1,1)$  proving assertion {\em a)}.

In order to prove {\em b)} we first note that  for all $i,r$ we have $C^{ir}\!(x^{\mbox{-}i})\in [C_1^r,C_N^r]$,
from which it follows that all the forward orbits of \eqref{ad} are bounded.
Moreover, since the Jacobian \eqref{jac22} has only one zero in each row and  column, 
then $J(x)$ is irreducible in the sense of \cite{smith}. Hence we may invoke \cite[Lemma 2.3(b)]{smith}
which asserts precisely that the existence of $T>0$ with either $\phi_T(x_0)<_s x_0$ or $\phi_T(x_0)>_s x_0$
implies that the orbit $\phi_t(x_0)$ converges to a rest point.
 \end{proof}

\vspace{1ex}
\begin{theorem} \label{TT3}Consider the mean field dynamics \eqref{ad} for a $2\times 2$ traffic game.
Then for almost all initial conditions $x_0\in\mathbb{R}^4$ the corresponding orbit $x(t)=\phi_t(x_0)$
converges to a rest point. Moreover, the union of the basins of attraction of all the steady states
is dense in $\mathbb{R}^4$.
\end{theorem}
\begin{proof}
We already noted that all forward orbits of \eqref{ad} are bounded, and also that there are  finitely many
rest points. Also, from Proposition \ref{mono} we know that 
\eqref{ad} is $K_s$-monotone and irreducible. Hence \cite[Theorem 2.5]{smith} asserts precisely
that the orbits $\phi_t(x_0)$ converge towards a rest point for almost all initial conditions $x_0\in\mathbb{R}^4$,
which establishes the first claim.

In order to prove the second assertion we use \cite[Theorem 2.6]{smith} which states that if $X_0$ is an open, 
bounded, positively invariant set for a $K_s$-monotone system, whose closure $\bar{X}_0$ contains finitely many
rest points $E$, then the union of the interiors of the basins of attraction $B_e$ of these rest points 
$\cup_{e\in E\cap \bar X_0}\mathop{\rm int}(B_e)$ is dense in $X_0$. In our setting we may apply this result to an open rectangle
$X_0=(-R,R)^4$ with $R>\max\{C^a_2,C^b_2\}$, which is easily seen to be positively invariant for the mean field dynamics  \eqref{ad}.
Then the union of the basins of attraction of the (finitely many) rest points  is dense in $(-R,R)^4$
and we  conclude by letting $R\to\infty$.
\end{proof}

According to \cite[Lemma 2.2]{smith}  the matrix $P_s (\frac{\partial {\phi}_t}{\partial x_0}(x_0)) P_s$ 
has all its entries strictly positive for all $t > 0$, and therefore \eqref{ad} cannot have an attracting 
closed orbit nor two points of an $\omega$-limit set related by ${\le}_s$.
Also, by the Perron-Frobenius Theorem \cite{meyer} the spectral radius of 
$P_s (\frac{\partial {\phi}_t}{\partial x_0}(x_0)) P_s$  is a positive simple eigenvalue 
strictly larger in modulus than all the remaining eigenvalues, and the corresponding 
eigenvector is positive. For further properties of totally positive matrices we refer to \cite{ando}. 

Let us also observe that the Jacobian $J(x)$ has an eigenvalue equal to $$s(J(x)) = \max \{\Re({\lambda}): \lambda \mbox{ eigenvalue of }J(x)\}.$$
Theorem 2.7 in \cite{smith} provides a set of equivalent conditions for the stability condition $s(J(x)) < 0$. In particular, condition (iv) in \cite[Theorem 2.7]{smith} matches Routh-Hurwitz criterion   \cite{murray},  while the Gerschgorin criterion \cite{morton} follows from condition  (ii) in \cite[Theorem 2.7]{smith}.

From a bifurcation analysis viewpoint, a rest point loses stability when a simple real eigenvalue crosses from negative to positive \cite{guckenheimer}. 
Therefore, locally stable Hopf bifurcations cannot occur for $K_s$-monotone systems, although unstable Hopf bifurcations cannot be excluded \cite{kuznetsov}.  
The only local bifurcation involving an exchange of stability that a steady state can be involved in is a steady state bifurcation. 
The next result provides additional information on the dynamics \eqref{ad} in the presence of an unstable equilibrium.
Recall that in this case there are 3 rest points $x_-\leq_s\bar x\leq_s x_+$ with $\bar x$ unstable and the other two  stable.

\vspace{1ex}
\begin{theorem}\label{TT5} Consider the case (c) in Theorem {\em \ref{TT2}} where $\psi$ has three fixed points $w_-\!<\!\bar w\!<\!w_+$.
Then the unstable equilibrium $\bar x$  is connected by two monotone heteroclinic orbits to the stable 
rest points $x_-$ and $x_+$. 
\end{theorem}
\begin{proof} The matrix $J(\bar x)$ is similar to $A=P_sJ(\bar x)P_s$ which is irreducible and has non-negative
off-diagonal entries. Hence for $\alpha>0$ large enough $\alpha I + A$ is non-negative and irreducible so that 
by the Perron-Frobenius Theorem we have that $s(\alpha I + A)=\alpha+s(J(\bar x))$ is a simple eigenvalue 
of $\alpha I + A$ with corresponding eigenvector $v\geq 0$. It follows that $\bar s=s(J(\bar x))$ is a simple 
eigenvalue of $J(\bar x)$ with eigenvector $\bar v=P_s v$, and since $\bar x$ is unstable we have $\bar s>0$. The conclusion 
then follows by a straightforward application of \cite[Theorem 2.8]{smith} which establishes the existence of two 
$C^1$ heteroclinic curves $y_+:[0,\infty)\to\mathbb{R}^4$
and $y_-:[0,\infty)\to\mathbb{R}^4$ such that
\begin{itemize}
\item[a)] $y_\pm(r)=\bar x \pm r \bar v +o(r)$ for $r\to 0$ and $y_\pm(r)\to x_\pm$ for $r\to\infty$,
\item[b)] for each $r\geq 0$ the orbits started from $y_\pm(r)$ follow these curves and are given 
by $\phi_t(y_\pm(r))=y_\pm(r e^{\bar s t})$ for all $t\geq 0$, and
\item[c)] for all $0\leq r_1\leq r_2$ we have $y_+(r_1)\leq_s y_+(r_2)$ and $y_-(r_2)\leq_s y_-(r_1)$.
\end{itemize}
\end{proof}

%This will be a fold bifurcation for maps \cite{kuznetsov}. On the other hand,  
% \cite[Proposition 2.9]{smith} states that if $x_1$ and $x_2$ are stable steady states such that $x_1 <_K x_2$ then there exists 
% at least one unstable steady state $x_0$ with $x_1 <_K x_0 <_K x_2$. Furthermore,  \cite[Proposition 2.10]{smith} states that 
% if $x_0$ is an unstable steady state, then its stable manifold contains two distinct points $x$ and $y$ with $x \; {\le}_K \; y$.

Even though a $K$-monotone system cannot have an attracting periodic orbit, when time is reversed these systems often do have attracting closed orbits. It can be shown that every closed orbit has a simple Floquet multiplier \cite{kuznetsov} which exceeds unity. This largest multiplier gives rise to a very unstable cylinder manifold associated with the closed orbit which has monotonicity properties \cite{smith}. In the specific case of \eqref{ad} we do not know whether the time reversed dynamics have an attracting closed orbit or not.

% STOCHASTIC CONVERGENCE
%
%{\color{blue}
%\subsection{Convergence of the stochastic dynamics}\label{secstoc}
%Using general results on stochastic approximation, from Theorem \ref{TT3} we can 
%deduce the convergence of the stochastic learning process.
%
%\vspace{1ex}
%\begin{theorem} Consider the stochastic adaptive dynamics \eqref{discrete_dynamics} for a $2\times 2$ traffic game,
%and assume that $\sum\lambda_n=\infty$ and $\sum\lambda_n^2<\infty$.
%Then for all initial conditions $x_0$ the process $x_n$ converges almost surely towards an equilibrium of the
%mean field dynamics.
%\end{theorem}
%\begin{proof} We begin by noting that the iterates $x_n$ are built by averaging the previous iterate with the 
%route travel times $C^r_{u^r_n}$. Since the latter are bounded between $C_1^r$ and $C_N^r$, it follows that 
%every trajectory $x_n$ remains bounded. 
%
%Let $\Lambda$ denote the $\omega$-limit set of $(x_n)_{n\in\mathbb{N}}$.
%Combining \cite[Proposition 4.1]{ben}, \cite[Proposition 4.2]{ben}, 
%and \cite[Theorem 5.7(i)]{ben} it follows that $\Lambda$ is an internally 
%chain transitive set for the mean field dynamics \eqref{ad}. 
%By \cite[Proposition 5.3]{ben} 
%this means that $\Lambda$ is a compact set which is invariant under the mean field dynamics 
%and which does not contain any proper attractor. 
%
%
%\end{proof}
%
%}

\subsection{The $2\times 2$ case with symmetric players}\label{sec44}

Theorem \ref{TT2} establishes three possible situations for the number and stability of the rest points 
of the mean field dynamics. However, it does not provide explicit conditions on the data of the 
traffic game to predict which situation will occur. We will show next that this is possible when the two 
players are identical with $\beta_1=\beta_2$. In the sequel we denote $\beta$ this common value.
Introducing the function $\theta(w)=\beta \varphi(w)$, the system \eqref{eq22prime} becomes 
\begin{equation}\label{eq22prime2}
\begin{array}{l}
w^1=\theta(w^2)\\
w^2=\theta(w^1)
\end{array}
\end{equation}
so that both $w^1$ and $w^2$ are  fixed points of $\psi(w)=\theta\circ \theta(w)$. 

Now, the function $\theta(w)$ decreases from $\beta(\kappa+\delta)$ at $w=-\infty$ to $\beta\kappa$ at $w=\infty$,
so that  it has a unique fixed point $\bar w=\theta(\bar w)$. 
This yields a unique symmetric solution $w^1=w^2=\bar w$  for \eqref{eq22prime2}, which corresponds to a  symmetric
rest point $\bar x$ with $\bar x^{1a}=\bar x^{2a}$ and $\bar x^{1b}=\bar x^{2b}$.

Note that if $\psi$ admits another fixed point $w^1\neq\bar w$ then $w^2=\theta(w^1)$ is also a fixed
point. Since $\theta$ is decreasing, if $w^1<\bar w$ then $w^2>\bar w$ and viceversa so that 
$\bar w$ is always  the centermost fixed point of $\psi$. Hence case {\em (b)} 
in Theorem \ref{TT2} cannot occur and $\psi$ has either one or three distinct fixed points.

\vspace{1ex}
\begin{theorem} \label{TT4}
Consider a  $2\times 2$ symmetric traffic game. 
Let $\bar x$ be the symmetric rest point associated with   the unique fixed point $\bar w$ of $\theta$.

\begin{itemize}
\item[(a)] If $\psi'(\bar w)\leq 1$ then $\bar w$ is the only fixed point of $\psi$ and
$\bar x$ is the unique rest point of \eqref{ad}. This rest point is stable if and only if 
$\psi'(\bar w)<1$.
\item[(b)] If $\psi'(\bar w)>1$ then $\psi$ has exactly three fixed points $w^1<\bar w< w^2$ with $w^1\!=\!g(w^2)$
and $w^2\!=\!g(w^1)$, and we have
$$\psi'(w^1)=\psi'(w^2)=\theta'(w^1)\theta'(w^2)\!<\!1.$$ 
The symmetric rest point $\bar x$ is unstable and there are two stable non-symmetric
equilibria: one defined by $\pi^{1a}=\rho(w^1), \pi^{2a}=\rho(w^2)$, and the other one with the identities of the players
exchanged.
\end{itemize}
\end{theorem}

\begin{proof} This follows from the previous comments and Theorem \ref{TT2}.
\end{proof}

\begin{figure}[!tbh]
\centering
\includegraphics[scale=0.375]{stab3.eps}
\includegraphics[scale=0.375]{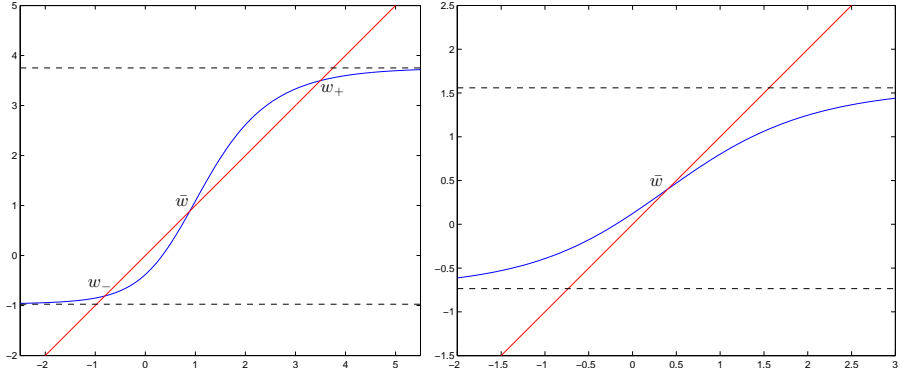}
\caption{Fixed points of $\psi$:  $\psi'(\bar w)>1$ (left) vs  $\psi'(\bar w)<1$ (right).} 
\label{fi:fp13}
\end{figure}

As seen above, in the symmetric case the number of rest points and their nature 
are completely determined by the condition $\psi'(\bar w)<1$. 
The latter can be checked directly 
from the data definining the game. More precisely,
denoting 
\begin{equation}\label{parc}
c=(C_1^a+C_2^a)-(C_1^b+C_2^b)
\end{equation}
we have the following characterization established in \cite{maldonado}.
For the sake of completeness we include a proof. 
We recall that the  hyperbolic arctangent function is defined by
 $\mathop{\rm arctanh} x = \frac{1}{2} \ln \left( \frac{1+x}{1-x} \right)$ for $x \in (-1,1)$.

\vspace{1ex}
\begin{proposition} \label{PPP} Consider the parameters $\mu=\beta\delta$ and $q=c/\delta$, as well as
the function $h(x)=x\,\mathop{\rm arctanh}\sqrt{1\!-\!x} \,+\sqrt{1\!-\!x}$.
Then  $\psi'(\bar w)<1$ if and only if  $|q|< h(4/\mu)$.
\end{proposition}

\begin{proof}  Noting that $\psi'(\bar w)=\theta'(\bar w)^2$  the condition $\psi'(\bar w)<1$ is
equivalent to $\theta'(\bar w)<-1$. Since $\theta'(\bar w) = - \beta\delta e^{\bar w} (1+e^{\bar w})^{-2}$,
denoting $\bar z=1+e^{\bar w}$ this can be written as
$\mu (\bar z-1)>\bar z^2$ which translates into the fact that $\bar z$ lies strictly between the roots 
of $z^2-\mu \bar z+\mu=0$, that is $z_-<\bar z<z_+$ where
\begin{equation}\label{roots}
z_\pm=\mbox{$\frac{1}{2}$}[\mu\pm\sqrt{\mu(\mu-4)}].
\end{equation}
Denoting $w_\pm=\ln( z_\pm-1)$ this is equivalent to $w_-<\bar w<w_+$.
Since $\bar w$ is the fixed point of $\theta(\cdot)$ and since this function is decreasing,
this  is in turn equivalent to the  two inequalities
$$\theta(w_-)>w_-\quad;\quad\theta(w_+)<w_+.$$
Now, observing that $\theta(w_\pm)=\beta\kappa+\beta\delta/z_\pm$, the latter can be written as
$$\ln(z_--1)-\frac{\beta\delta}{z_-}<\beta\kappa<\ln(z_+-1)-\frac{\beta\delta}{z_+}.$$
Replacing the expression \eqref{roots} for $z_\pm$ and denoting
$$f_\pm(\mu)=\frac{1}{\mu}\ln \left( \frac{\mu\pm\sqrt{\mu(\mu-4)}}{2}-1 \right) -\frac{2}{\mu\pm\sqrt{\mu(\mu-4)}}$$
these inequalities can be expressed as
\begin{equation}\label{last}
f_-(\mu)<\frac{\kappa}{\delta}<f_+(\mu).
\end{equation}
Now, by direct computation one can check that $f_-(\mu)+f_+(\mu)=-1$ so that,
setting $g(\mu)=2f_+(\mu)+1$,  we may restate \eqref{last} in the symmetric form
$$-g(\mu)<1+\frac{2\kappa}{\delta}<g(\mu).$$
The conclusion follows by noting that $1+2\kappa/\delta=q$ and
$g(\mu)=h(4/\mu)$.
\end{proof}

Figure \ref{fi:bp} illustrates the stability region for the $2 \times 2$ symmetric game. 
This plot confirms that for $\mu =\beta\delta< 4$, which is equivalent to the condition $\nu\delta^2<16$
 in the comment after Theorem \ref{TT1}, the symmetric equilibrium $\bar x$ is the unique rest point and it is 
stable. When $\theta'(\bar w)$ crosses the value $-1$
the symmetric steady state loses stability and two new stable equilibria appear. This is
due to the crossing of the imaginary axis by a simple negative eigenvalue. Bifurcation theory for maps 
also predicts this situation which is indicative of a fold bifurcation (see \cite{kuznetsov}).
\begin{figure}[!tbh]
\centering
\includegraphics[scale=0.6]{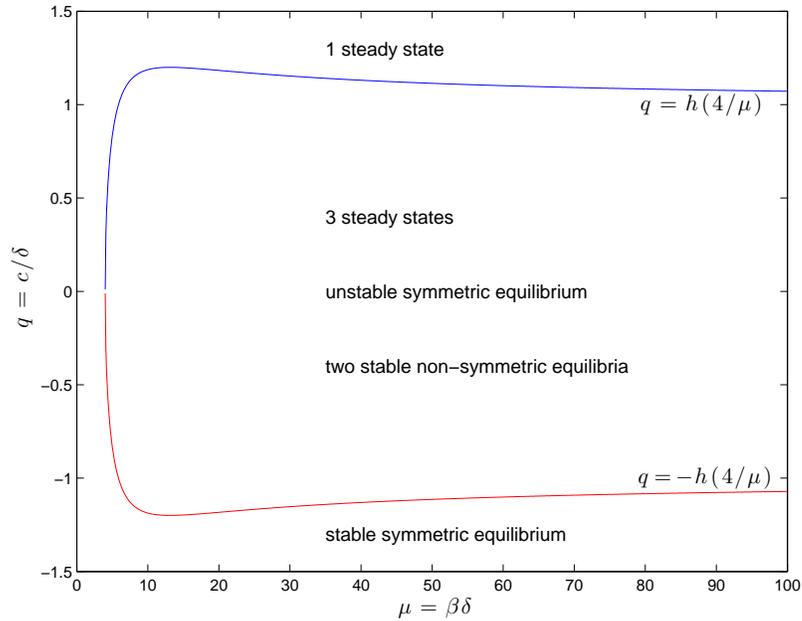}
\caption{Stability region for a $2 \times 2$ symmetric game in the $\mu$-$q$ parameters.} 
\label{fi:bp}
\end{figure}

\section{Summary and future work}
In this paper, we gave a necessary and sufficient condition for the stability of equilibria 
for a $2\times 2$ symmetric traffic game. The main tool used was the Routh-Hurwitz 
stability criterion for dynamical systems. We also exploited  the results for $K$-monotone systems 
to establish the asymptotic convergence of the mean field dynamics.

With respect to Proposition \ref{PPP} we conjecture that the condition $|q|<h(4/\beta\delta)$ for having a unique stable equilibrium
might still hold for the case of non-symmetric players provided that 
we take $\beta=\sqrt{\beta_1\beta_2}$.

The extension to  the general case with $M$ routes and $N$ players seems much more difficult.
%In \cite{cominetti}, the continuous dynamical system of size $NM$, for $x^{is}$, the pay-off perception of the $i$-th player (among $N$ players) of the $s$-th pure strategy (for simplicity assuming that each player has the same total number $M$ of pure strategies), is obtained by an averaging process of the asymptotic discrete time random process. The Routh-Hurwitz necessary and sufficient criterion can be applied to the context of arbitrary number of players $N$ and pure strategies $M$. Even though the special structure of the Jacobian of those systems of size $NM$, somehow simplifies the algebraic expressions of the Routh-Hurwitz sub-determinants, we were not able to obtain simple expressions for all the sub-determinant constraints. Moreover, we were able to show that the Jacobian structure for games with more than two players or more than two pure routes are not order-preserving dynamical systems. Hence, some other techniques are necessary to study the global dynamic behavior of the general $N \times M$ case.
However our numerical simulations show that the dynamics still converge to equilibria
and exhibit a certain regularity in the structure of the attractors. Namely, for the case 
with identical players $\beta_i\equiv\beta$, there is a unique {\em symmetric} equilibrium 
in which $\bar x^{ir}=\bar x^{jr}$ for all routes $r$ and different players $i,j$. When $\beta$ is small
this is the only rest point and a global attractor, whereas for larger $\beta$ the symmetric equilibrium becomes unstable and there appear non-symmetric rest points $(x^i)_{i\in I}$
with $x^i\neq x^j$ for some $i\neq j$. However, these non-symmetric 
equilibria exhibit clusters of players with the same estimate vector $x^i$'s and, interestingly, 
the number of players in each cluster is stable across the different equilibria.
It would be interesting to understand the reasons for such stable structure. 

Another possible line of research is to investigate the extension of the results 
from routing games to general 2-player games and beyond. Although some asymptotic results 
were already described in \cite{cominetti} for the adaptive dynamics in general finite games,  
they hold only in the regime where the Logit parameters $\beta_i$ are small. Since
the results in this paper strongly exploit the symetric structure of the Jacobian matrix 
in \S{3}, it is unclear if and how the arguments could be adapted to general games.

\end{document}